\documentclass{amsart}

\usepackage{url}
\usepackage{tikz}
\usetikzlibrary{shapes.geometric}
\usetikzlibrary{decorations.pathmorphing}

\theoremstyle{plain}
\newtheorem{theorem}{Theorem}

\theoremstyle{definition}

\theoremstyle{remark}

\newtheorem{observation}{Observation}

\pgfdeclaredecoration{complete sines}{initial}
{
    \state{initial}[
        width=+0pt,
        next state=sine,
        persistent precomputation={\pgfmathsetmacro\matchinglength{
            \pgfdecoratedinputsegmentlength / int(\pgfdecoratedinputsegmentlength/\pgfdecorationsegmentlength)}
            \setlength{\pgfdecorationsegmentlength}{\matchinglength pt}
        }] {}
    \state{sine}[width=\pgfdecorationsegmentlength]{
        \pgfpathsine{\pgfpoint{0.25\pgfdecorationsegmentlength}{0.5\pgfdecorationsegmentamplitude}}
        \pgfpathcosine{\pgfpoint{0.25\pgfdecorationsegmentlength}{-0.5\pgfdecorationsegmentamplitude}}
        \pgfpathsine{\pgfpoint{0.25\pgfdecorationsegmentlength}{-0.5\pgfdecorationsegmentamplitude}}
        \pgfpathcosine{\pgfpoint{0.25\pgfdecorationsegmentlength}{0.5\pgfdecorationsegmentamplitude}}
}
    \state{final}{}
}

\newcommand{\actionedge}{\draw[thick,decorate,decoration={complete sines,amplitude=3pt}]}
\newcommand{\responseedge}{\draw[dashed]}
\newcommand{\vertexmarkone}{\node[inner sep=0.5em,shape=circle,draw,fill=none] ()}
\newcommand{\vertexmarktwo}{\node[inner sep=0.6em,shape=diamond,draw,fill=none] ()}
\newcommand{\vertexmarkthree}{\node[inner sep=0.6em,shape=regular polygon, regular polygon sides=4,draw,fill=none] ()}

\newcommand{\cf}{\mathcal{F}}

\DeclareMathOperator{\ex}{ex}
\DeclareMathOperator{\sat}{sat}

\title{An upper bound on the extremal version of Hajnal's triangle-free game}

\author{Csaba Bir\'o}
\address[Csaba Bir\'o]{Department of Mathematics, University of Louisville, Louisville, KY 40292}
\email{csaba.biro@louisville.edu}

\author{Paul Horn}
\address[Paul Horn]{Department of Mathematics, University of Denver, Denver, CO 80208}
\email{paul.horn@du.edu}

\author{D. Jacob Wildstrom}
\address[D. Jacob Wildstrom]{Department of Mathematics, University of Louisville, Louisville, KY 40292}
\email{dwildstr@erdos.math.louisville.edu}

\begin{document}

\begin{abstract}
  A game starts with the empty graph on $n$ vertices, and two player
  alternate adding edges to the graph. Only moves which do not create
  a triangle are valid.  The game ends when a maximal triangle-free
  graph is reached. The goal of one player is to end the game as soon
  as possible, while the other player is trying to prolong the
  game. With optimal play, the length of the game (number of edges
  played) is called the $K_3$ game saturation number.

  In this paper we prove an upper bound for this number.
\end{abstract}

\maketitle

\section{Introduction}

Hajnal proposed the following game: let $n$ be a positive integer, and
let $G_0$ be the empty graph on $n$ vertices. Two players alternate
adding edges between non-adjacent vertices. The first to create a
triangle loses. Note that there is no difference between the edges
played by either player. The question is, of course, who has a winning
strategy for a given $n$.

The answer is only known for small values of $n$. Namely, the first player wins
for $n=6, 10, 12, 13, 14, 15, 16$, and the second player wins for
$n=3,4,5,7,8,9,11$. Values up to $n=11$ were known in the early 90's, and
Cater, Harary, and Robinson \cite{Cat-Har-Rob-01} used computers to solve the
$n=12$ case. With heavy use of computers, Pra\l at \cite{Pra-10} settled the
cases of $n=13,14,15$, and then Gordinowicz and Pra\l at \cite{Gor-Pra-12}
settled the case of $n=16$.

F\"uredi, Reimer, and Seress \cite{Fur-Rei-Ser-91} proposed a
variation of this game. In this variant, instead of complete victory
belonging to a single player, the game has a score which each player
tries to manipulate. The two players play until a maximal
triangle-free graph is achieved, and the score is the number of edges
played. The player who moves first is trying to maximize this number,
while the other one is trying to minimize it.  Assuming perfect
strategy from both players, there is a well-defined function that
assigns the score of each game (under perfect play) to each positive
integer $n$. F\"uredi, Reimer, and Seress proved the following
theorem.
\begin{theorem}
The score is at least $(n\log n)/2-2n\log\log n+O(n)$.
\end{theorem}

In \cite{Fur-Rei-Ser-91} and \cite{Ser-92}, the authors cite personal
communication with Paul Erd\H os, who is claimed to have proven an upper bound
for the score under perfect play of $n^2/5$. This proof is probably
lost.\footnote{We asked all three authors who claimed personal communication
with Erd\H os what the proof was, and none of them remembered. Seress suggested
that it was a short sketch of an argument, that they all agreed that it worked,
but it was not deemed to be worthy of publication.} One of the original
motivation of the authors was the reconstruction of Erd\H os's proof. After
some shorter, and later, longer attempts of proofs and improvements, which
turned out to be just slightly wrong, we had to settle on a bit weaker bound
presented in this paper. However, it is all but certain, that this proof is not
what Erd\H os had in mind.

\subsection{Competitive optimization}

The problem discussed in this article is only one of a large area of research,
called competitive optimization, with strong relations to extremal graph
theory.

Let $\cf$ be a family of graphs. If a graph $G$ has no subgraph from
$\cf$, but adding an edge between any pair of non-adjacent vertices of
$G$ would create a subgraph from $\cf$, then we say that $G$ is
\emph{$\cf$-saturated}. If $\cf$ contains a single graph $H$, then we
conventionally refer to $H$-saturation rather than $\{H\}$
saturation. The ``game board'' at the end of the extremal version of
Hajnal's game is an $K_3$-saturated graph. The well-studied
\emph{Tur\'an-number}, written $\ex(\cf;n)$ is the maximum number of
edges in an $\cf$-saturated graph on $n$ vertices. The minimum number
of edges of an $\cf$-saturated graph on $n$ vertices is called the
\emph{saturation number}, denoted by $\sat(\cf;n)$.

A game version of these numbers is defined as follows. A two person game is
played, starting on the empty graph on $n$ vertices, and the two players add
edges to the graph, alternating, until an $\cf$-saturated graph is reached. The
goal of the player who moves first is to maximize the number of edges in this
final graph, while the second player is trying to minimize the number of edges.
The number of edges played under optimal play is called the game saturation
number of $\cf$, denoted by $\sat_g(\cf;n)$.

An even more general version of the game saturation number where the
``playable'' edges are limited by an initial ``host'' graph (which is $K_n$ for
us, as all edges are playable) was defined in full generality by West
\cite{West}. Many recent results were obtained in specific settings, see
e.g.~\cite{Car-Kin-Rei-Wes-14-u} and \cite{Lee-Rie-14-u}.

Strictly speaking, there are two kinds of game saturation numbers, and the
other kind, denoted by $\sat_g'(\cf;n)$ is if the first player is the one
minimizing the number of edges in the final graph. While these two versions may
differ significantly (see e.g.~\cite{Car-Kin-Rei-Wes-14-u}), our proof works in
both cases, so our (asymptotic) result applies to both types of game saturation
number.  For convenience, we will assume that the minimizing player is starting
the game, though the reader will see that it does not really matter.

\subsection{Some proof ideas}

It is obvious that for any family $\cf$, we have $\sat(\cf;n)\leq \sat_g(\cf;n)\leq
\ex(\cf;n)$. These provide the trivial bounds $n-1$ and $n^2/4$ for the
triangle-free game. While there are short proofs that provide modest
improvements on the lower bound, we are not aware of any short proof of any
(multiplicative) improvement on the upper bound. We make an observation that
will be crucial for our proof.

\begin{observation}
There are at most $10$ edges between any pair of $C_5$'s in a triangle-free
graph. Also, between a vertex and a $C_5$, there are at most $2$ edges in a
triangle-free graph.
\end{observation}

In addition to the fact that a $C_5$ itself is $K_3$-saturated, if one
could build just one $C_5$ using, say, $k$ vertices of the graph
(i.e.~all other vertices will have degree $0$), then this strategy may
be repeated on empty vertices, and after the whole game board is
filled with these building blocks, the number of edges is at most
\[
\frac{\left(\frac{n}{k}\right)^2}{2}\cdot 10+
\frac{n^2\left(\frac{k-5}{k}\right)^2}{4}+
n\left(\frac{k-5}{k}\right)\cdot\frac{n}{k}\cdot 2+
o(n^2)\approx \frac{k^2-2k+5}{4k^2}n^2.
\]
(Essentially the same calculation is detailed out for $k=11$ at the end of the
proof of our main theorem.)

Here $k\geq 5$. If $k$ is very large, the this bound gets very close to
$n^2/4$, but with any finite $k$, it provides a multiplicative improvement over
the trivial bound. The only problem is that building even one $C_5$ on finitely
many vertices seem not be that easy. It is possible, as our proof will
show (and more), but there are some technical difficulties.

If, by some miracle, $k=5$ is possible, we would get the Erd\H os bound $\approx
n^2/5$. In fact, if we one could just build an almost perfect $C_5$-factor of
the graph, with $o(n^2)$ vertices not in the $C_5$-factor, the $n^2/5$ bound
would follow (asymptotically). Unfortunately, no strategy can guarantee such a
$C_5$-factor. The maximizing player can build a star on about
$n/2$ vertices, and this can not be stopped. Then any $C_5$ may only use two
vertices of the star leaves, so it is not possible to create more than roughly
$n/6$ copies of $C_5$'s.

\section{The $C_5$-building strategy}

Still the core of this technique for ending the game as soon as possible is
to build as many disjoint cycles of length 5 as possible; since a
$C_5$ is the minimal nonbipartite triangle-free graph, the
incorporation of as many vertices as possible into $C_5$ subgraphs
minimizes the number of vertices which can be incorporated into a
large, balanced bipartite graph.

However, forcing the construction of $C_5$ subgraphs when confronted
with an opponent who is trying to prevent the construction of a $C_5$
is not easy. It is easy to build a path on 4 vertices in a way which
cannot be prevented by the opponent, but the attempt to extend this
into a $C_5$ can be stymied by the opponent prematurely closing it
into a $C_4$. For that reason, a more complicated technique in which
two parallel paths are constructed and then joined is necessary.

\begin{theorem}
  Starting with $n$ vertices and no edges, there is a sequence of
  moves by one player which, regardless of the other player's actions,
  leads to $\lfloor\frac{n-2}{11}\rfloor$ disjoint $C_5$s being
  constructed.
  \label{theorem:C5-construction}
\end{theorem}
\begin{proof}
  We shall show that a particular $C_5$-construction procedure can be
  repeated several times. During the course of this procedure, let us
  denote the set of vertices which are not yet vertices of a
  constructed $C_5$ as $U$. Let the ``count'' of $U$ be calculated
  according to the following method: $U$ has count equal to the number
  of different components of the graph in which the vertices of $U$
  lie, plus whichever is largest of the following measures of partial
  progress towards the next $C_5$-construction:
  \begin{itemize}
  \item $5$ if there is a $P_3$ subgraph among the vertices of $U$ and
    another vertex of $U$ in the same component as the $P_3$ at a
    distance of at least 3 from an endpoint of the $P_3$.
  \item $4$ if there is a $P_3$ subgraph among the vertices of $U$.
  \item $3$ if there is a $P_2$ subgraph among the vertices of $U$ and
    another vertex of $U$ in the same component as the $P_2$ at a
    distance of at least 3 from an endpoint of the $P_2$.
  \item $2$ if there is a $P_2$ subgraph among the vertices of $U$.
  \item $1$ if there are two vertices of $U$ at a distance of at least
    3 within a single component of the graph.
  \item $0$ otherwise.
  \end{itemize}
  The construction procedure below will be shown to be implementable
  whenever the count exceeds 13, and will have a net effect of
  reducing the count by at most 11; thus, the procedure may be
  implemented $\lfloor\frac{n-2}{11}\rfloor$ times before the count
  becomes too low to repeat it.

  The $C_5$-constructing player will begin by building a $P_4$ among
  vertices in $U$. If there is already a $P_3$ in the graph, doing so
  will take one move and reduce the number of components in the graph
  by 1. If there is already a $P_2$ in the graph and a vertex at
  distance 3 from the $P_2$, this construction takes two moves, but
  will only reduce the number of components in the graph by 1, as the
  first edge may be added between vertices in a single component. If
  there is already a $P_2$ in the graph and no such other vertex, this
  construction takes two moves, and reduces the number of components
  in the graph by 2. if there are no usable subgraphs then this
  construction step takes 3 moves; if two vertices are at a distance 2
  in a single component it will only reduce the number of components
  by 2, but otherwise it reduces the number of components by 3. Thus,
  we may note that the sum of the reduction in component count and
  number of moves taken will always be equal to six minus the
  above-mentioned ``count bonus''; since every move used allows the
  opponent an opportunity to move and possibly reduce the count, we
  may conclude that the $P_4$-construction stage, if completed without
  opponent interference, will result in a reduction of the count by at most 6,
  as at this point there is no guarantee that count-bonus structures
  remain.

  We may note that it is impossible for the other player to obstruct
  this construction, since adding edges among vertices in distinct
  components is always permitted; in the course of constructing this
  $P_4$, either the opponent will have used their moves in
  count-reduction, reducing the count by at most 6 as mentioned above,
  or have used one move in converting this $P_4$ into a $C_4$, and
  effecting a count reduction of no more than 5. We will address these possibilities in two subsections.

  \subsection{Opponent does not create a $C_4$}

  Now we shall create a second $P_3$ using vertices from $U$ in
  distinct components. This $P_3$ requires the addition of 2 edges,
  during which the opponent may add two edges. These edges may include
  external edges, a single edge converting the original $P_4$ to a
  $C_4$, edges between the $P_4$ and developing $P_3$, or edges to
  from the $P_3$ to an external vertex adjacent to the $P_4$. Note
  that we may guarantee that one of the endpoints of this new $P_3$ is
  ``untouched'' as such: after we have added one of our two edges, at
  most one vertex of this $P_3$ will be ``touched'', and we may add
  the new edge such that the touched vertex is the middle, and thus
  both endpoints are untouched, after which the opponent has the
  opportunity to touch at most one of the endpoints.

  \begin{figure}
    \centering
    \begin{tikzpicture}[nodes={shape=circle,fill,draw}]
      \node (u1) at (0,0) {};
      \node (u2) at (0,1) {};
      \node (u3) at (0,2) {};
      \node (u4) at (0,3) {};
      \draw (u1)--(u2)--(u3)--(u4);
      \node (v1) at (1.5,1) {};
      \node (v2) at (1.5,2) {};
      \draw (v1)--(v2);
      \actionedge (u1)--(v1);
      \responseedge (u4)--(v1);
      \responseedge (u3)--(v2);
      \node[fill=none,draw=none] () at (0.75,-0.5) {$G_1$};
    \end{tikzpicture}
    \quad\quad\quad
    \begin{tikzpicture}[nodes={shape=circle,fill,draw}]
      \node (u1) at (0,0) {};
      \node (u2) at (0,1) {};
      \node (u3) at (0,2) {};
      \node (u4) at (0,3) {};
      \draw (u1)--(u2)--(u3)--(u4);
      \node (v1) at (1.5,1) {};
      \node (v2) at (1.5,2) {};
      \draw (v1)--(v2);
      \actionedge (u2)--(v1);
      \responseedge (u4)--(v2);
      \node[fill=none,draw=none] () at (0.75,-0.5) {$G_2$};
    \end{tikzpicture}

    \caption{There are two ways that the opponent could place an edge,
      shown as wavy lines, between the $P_4$ and an in-progress
      $P_3$; in response any of the dashed edges will complete a
      $C_5$.}
    \label{fig:P4-and-P2-connected}
  \end{figure}
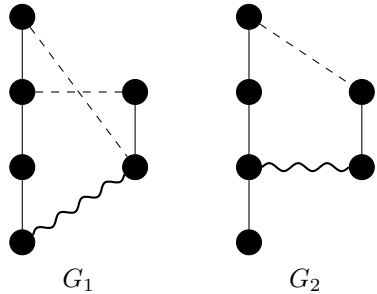

  Addressing the possible interferential opponent actions
  individually, the easiest one to dispense with is the prospect that
  the opponent adds an edge between the constructed $P_4$ and the
  under-construction $P_3$, which will result in a construct
  isomorphic to one of the two shown in
  Figure~\ref{fig:P4-and-P2-connected}. Since in adding the edge the
  opponent will have used the only opportunity to create obstructions
  relating to the vertices in the $P_3$, we know that any edge
  incident on the two vertices of the $P_3$ does not form a $K_3$
  unless it does so with edges known to us; thus the addition of any
  of the dashed edges in Figure~\ref{fig:P4-and-P2-connected}
  will complete a $C_5$.

  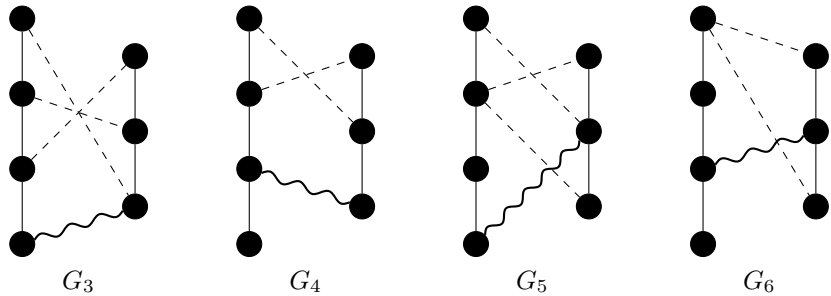
\begin{figure}
    \centering
    \begin{tikzpicture}[nodes={shape=circle,fill,draw}]
      \node (u1) at (0,0) {};
      \node (u2) at (0,1) {};
      \node (u3) at (0,2) {};
      \node (u4) at (0,3) {};
      \draw (u1)--(u2)--(u3)--(u4);
      \node (v1) at (1.5,0.5) {};
      \node (v2) at (1.5,1.5) {};
      \node (v3) at (1.5,2.5) {};
      \draw (v1)--(v2)--(v3);
      \actionedge (u1)--(v1);
      \responseedge (u4)--(v1);
      \responseedge (u3)--(v2);
      \responseedge (u2)--(v3);
      \node[fill=none,draw=none] () at (0.75,-0.5) {$G_3$};
    \end{tikzpicture}
    \quad\quad\quad
    \begin{tikzpicture}[nodes={shape=circle,fill,draw}]
      \node (u1) at (0,0) {};
      \node (u2) at (0,1) {};
      \node (u3) at (0,2) {};
      \node (u4) at (0,3) {};
      \draw (u1)--(u2)--(u3)--(u4);
      \node (v1) at (1.5,0.5) {};
      \node (v2) at (1.5,1.5) {};
      \node (v3) at (1.5,2.5) {};
      \draw (v1)--(v2)--(v3);
      \actionedge (u2)--(v1);
      \responseedge (u4)--(v2);
      \responseedge (u3)--(v3);
      \node[fill=none,draw=none] () at (0.75,-0.5) {$G_4$};
    \end{tikzpicture}
    \quad\quad\quad
    \begin{tikzpicture}[nodes={shape=circle,fill,draw}]
      \node (u1) at (0,0) {};
      \node (u2) at (0,1) {};
      \node (u3) at (0,2) {};
      \node (u4) at (0,3) {};
      \draw (u1)--(u2)--(u3)--(u4);
      \node (v1) at (1.5,0.5) {};
      \node (v2) at (1.5,1.5) {};
      \node (v3) at (1.5,2.5) {};
      \draw (v1)--(v2)--(v3);
      \actionedge (u1)--(v2);
      \responseedge (u4)--(v2);
      \responseedge (u3)--(v3);
      \responseedge (u3)--(v1);
      \node[fill=none,draw=none] () at (0.75,-0.5) {$G_5$};
    \end{tikzpicture}
    \quad\quad\quad
    \begin{tikzpicture}[nodes={shape=circle,fill,draw}]
      \node (u1) at (0,0) {};
      \node (u2) at (0,1) {};
      \node (u3) at (0,2) {};
      \node (u4) at (0,3) {};
      \draw (u1)--(u2)--(u3)--(u4);
      \node (v1) at (1.5,0.5) {};
      \node (v2) at (1.5,1.5) {};
      \node (v3) at (1.5,2.5) {};
      \draw (v1)--(v2)--(v3);
      \actionedge (u2)--(v2);
      \responseedge (u4)--(v1);
      \responseedge (u4)--(v3);
      \node[fill=none,draw=none] () at (0.75,-0.5) {$G_6$};
    \end{tikzpicture}

    \caption{There are four ways that the opponent could place an edge,
      shown in wavy lines, between the $P_4$ and a completed $P_3$; in
      response any of the dashed edges will complete a $C_5$.}
    \label{fig:P4-and-P3-connected}
   \end{figure}

   If, on the other hand, the opponent connects the $P_4$ and a
   just-completed $P_3$, it is possible that they have already spent a
   previous turn creating complications for one of the vertices in the
   $P_3$. These possibilities are illustrated in
   Figure~\ref{fig:P4-and-P3-connected}, and in each case, there are
   at least two vertices in the $P_3$ which can serve as termini for
   edges which form a $C_5$. Since the opponent has at most one
   opportunity to construct new edges incident to the $P_3$, they will
   be able to render at most one of those termini unusable. Thus, when
   the opponent builds an edge between the $P_3$ and $P_4$, we are
   guaranteed the ability to complete a $C_5$ in one more step.
   
   If, on the other hand, the opponent spends the two moves during the
   $P_3$ in other ways, we become responsible for building the edge
   between these two paths ourselves. If the opponent has not
   interfered with our construction during these steps, we may
   ourselves add the edge depicted with a wavy line in the first case
   of Figure~\ref{fig:P4-and-P3-connected} and then, with three
   possible $C_5$-completing edges among six distinct vertices, cannot
   be stopped.
   
   If the opponent does interfere, we know the nature of the
   interference: since we have addressed the possibility of an
   opponent-added edge between the paths, and since the only
   preventative to our adding edges freely is that a vertex in the
   $P_3$ and a set of non-adjacent vertices in the $P_4$ might
   possibly be mutually adjacent to a third vertex; the opponent has
   two opportunities to do this, and, as previously noted, the
   construction of the $P_3$ can be dynamically adjusted to ensure
   that one of the endpoints of the $P_3$ is untouched. Considering
   only the most restrictive choices of vertices in the $P_4$ which
   are mutually adjacent to some vertex, there are six possible
   scenarios up to isomorphism which may arise, depicted in
   Figure~\ref{fig:P4-and-P3-with-constraints}. Although the $P_4$ may
   be arbitrarily highly adjacent to external vertices, as depicted by
   the rings around each vertex, the $P_3$ was constructed with
   initially clean vertices, and so at most two external adjacencies
   can be introduced by the opponent while we are building the
   $P_3$. In each scenario, the move depicted with a wavy line
   produces two legal $C_5$-completing moves which do have nonadjacent
   endpoints in the $P_3$, so the opponent will be able to render at
   most one of them illegal, resulting in successful completion of a
   $C_5$.

   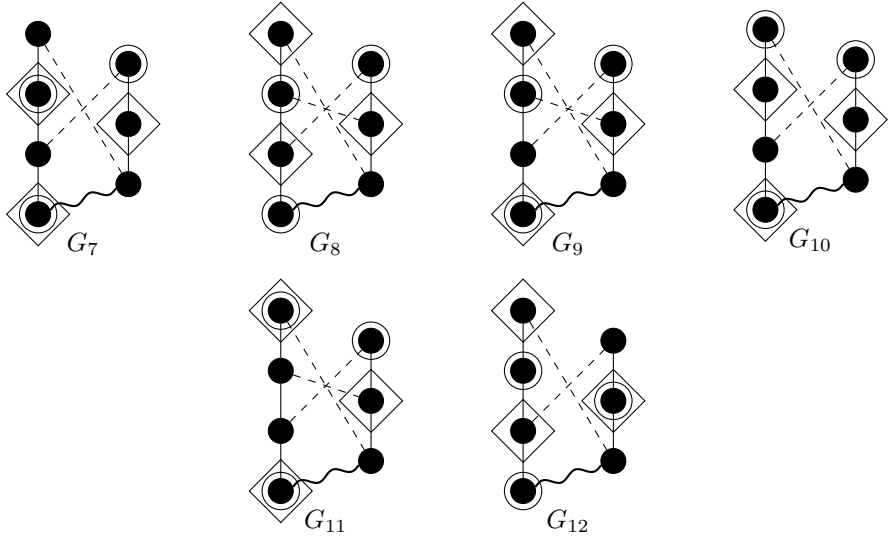
\begin{figure}
     \centering
     \begin{tikzpicture}[scale=0.8,nodes={shape=circle,fill,draw}]
       \node (u1) at (0,0) {};
       \node (u2) at (0,1) {};
       \node (u3) at (0,2) {};
       \node (u4) at (0,3) {};
       \draw (u1)--(u2)--(u3)--(u4);
       
       \vertexmarkone at (u1) {};
       \vertexmarkone at (u3) {};
       
       \vertexmarktwo at (u1) {};
       \vertexmarktwo at (u3) {};
       
       \node (v1) at (1.5,0.5) {};
       \node (v2) at (1.5,1.5) {};
       \node (v3) at (1.5,2.5) {};
       \draw (v1)--(v2)--(v3);

       \vertexmarkone at (v3) {};
       \vertexmarktwo at (v2) {};

       \actionedge (u1)--(v1);
       \responseedge (u4)--(v1);
       \responseedge (u2)--(v3);
      \node[fill=none,draw=none] () at (0.75,-0.5) {$G_7$};
     \end{tikzpicture}
     \quad\quad\quad
     \begin{tikzpicture}[scale=0.8,nodes={shape=circle,fill,draw}]
       \node (u1) at (0,0) {};
       \node (u2) at (0,1) {};
       \node (u3) at (0,2) {};
       \node (u4) at (0,3) {};
       \draw (u1)--(u2)--(u3)--(u4);
       
       \vertexmarkone at (u1) {};
       \vertexmarkone at (u3) {};
       
       \vertexmarktwo at (u2) {};
       \vertexmarktwo at (u4) {};
       
       \node (v1) at (1.5,0.5) {};
       \node (v2) at (1.5,1.5) {};
       \node (v3) at (1.5,2.5) {};
       \draw (v1)--(v2)--(v3);

       \vertexmarkone at (v3) {};
       \vertexmarktwo at (v2) {};
       \actionedge (u1)--(v1);
       \responseedge (u2)--(v3);
       \responseedge (u3)--(v2);
       \responseedge (u4)--(v1);
      \node[fill=none,draw=none] () at (0.75,-0.5) {$G_8$};
     \end{tikzpicture}
     \quad\quad\quad
     \begin{tikzpicture}[scale=0.8,nodes={shape=circle,fill,draw}]
       \node (u1) at (0,0) {};
       \node (u2) at (0,1) {};
       \node (u3) at (0,2) {};
       \node (u4) at (0,3) {};
       \draw (u1)--(u2)--(u3)--(u4);
       
       \vertexmarkone at (u1) {};
       \vertexmarkone at (u3) {};
       
       \vertexmarktwo at (u1) {};
       \vertexmarktwo at (u4) {};
       
       \node (v1) at (1.5,0.5) {};
       \node (v2) at (1.5,1.5) {};
       \node (v3) at (1.5,2.5) {};
       \draw (v1)--(v2)--(v3);

       \vertexmarkone at (v3) {};
       \vertexmarktwo at (v2) {};
       \actionedge (u1)--(v1);
       \responseedge (u2)--(v3);
       \responseedge (u3)--(v2);
       \responseedge (u4)--(v1);
      \node[fill=none,draw=none] () at (0.75,-0.5) {$G_9$};
     \end{tikzpicture}
     \quad\quad\quad
     \begin{tikzpicture}[scale=0.8,nodes={shape=circle,fill,draw}]
       \node (u1) at (0,0) {};
       \node (u2) at (0,1) {};
       \node (u3) at (0,2) {};
       \node (u4) at (0,3) {};
       \draw (u1)--(u2)--(u3)--(u4);
       
       \vertexmarkone at (u1) {};
       \vertexmarkone at (u4) {};
       
       \vertexmarktwo at (u1) {};
       \vertexmarktwo at (u3) {};
       
       \node (v1) at (1.5,0.5) {};
       \node (v2) at (1.5,1.5) {};
       \node (v3) at (1.5,2.5) {};
       \draw (v1)--(v2)--(v3);

       \vertexmarkone at (v3) {};
       \vertexmarktwo at (v2) {};

       \actionedge (u1)--(v1);
       \responseedge (u2)--(v3);
       \responseedge (u4)--(v1);
      \node[fill=none,draw=none] () at (0.75,-0.5) {$G_{10}$};
     \end{tikzpicture}
     \quad\quad\quad
     \begin{tikzpicture}[scale=0.8,nodes={shape=circle,fill,draw}]
       \node (u1) at (0,0) {};
       \node (u2) at (0,1) {};
       \node (u3) at (0,2) {};
       \node (u4) at (0,3) {};
       \draw (u1)--(u2)--(u3)--(u4);
       
       \vertexmarkone at (u1) {};
       \vertexmarkone at (u4) {};
       
       \vertexmarktwo at (u1) {};
       \vertexmarktwo at (u4) {};
       
       \node (v1) at (1.5,0.5) {};
       \node (v2) at (1.5,1.5) {};
       \node (v3) at (1.5,2.5) {};
       \draw (v1)--(v2)--(v3);

       \vertexmarkone at (v3) {};
       \vertexmarktwo at (v2) {};

       \actionedge (u1)--(v1);
       \responseedge (u2)--(v3);
       \responseedge (u3)--(v2);
       \responseedge (u4)--(v1);
      \node[fill=none,draw=none] () at (0.75,-0.5) {$G_{11}$};
     \end{tikzpicture}
     \quad\quad\quad
     \begin{tikzpicture}[scale=0.8,nodes={shape=circle,fill,draw}]
       \node (u1) at (0,0) {};
       \node (u2) at (0,1) {};
       \node (u3) at (0,2) {};
       \node (u4) at (0,3) {};
       \draw (u1)--(u2)--(u3)--(u4);
       
       \vertexmarkone at (u1) {};
       \vertexmarkone at (u3) {};
       
       \vertexmarktwo at (u2) {};
       \vertexmarktwo at (u4) {};
       
       \node (v1) at (1.5,0.5) {};
       \node (v2) at (1.5,1.5) {};
       \node (v3) at (1.5,2.5) {};
       \draw (v1)--(v2)--(v3);

       \vertexmarkone at (v2) {};
       \vertexmarktwo at (v2) {};

       \actionedge (u1)--(v1);
       \responseedge (u2)--(v3);
       \responseedge (u4)--(v1);
      \node[fill=none,draw=none] () at (0.75,-0.5) {$G_{12}$};
     \end{tikzpicture}
     \caption{The opponent's possible interferences with vertices in
       the $P_3$ is depicted symbolically by surrounding the vertex
       with a ring; matching rings indicate mutual adjacency to a
       third vertex, prohibiting edges between vertices with matching
       rings; advantageous moves and resulting $C_5$ constructions are
       labeled with wavy and dashed lines respectively.}
     \label{fig:P4-and-P3-with-constraints}
   \end{figure}
   
   \subsection{Opponent creates a $C_4$}
   
   Now we address those cases in which the opponent, in
   response to the formation of the $P_4$, closes it into a
   $C_4$. Under such a circumstance, we will extend the $P_3$ into a
   $P_4$. If the opponent chooses to place an edge between the $C_4$
   and an under-construction $P_4$, then we will be faced with one of
   the situations described in
   Figure~\ref{fig:C4-and-P4-interrupted}. In each of these situations
   one of the $C_5$-completion edges must still be a valid move: if we
   have only completed a $P_2$, the opponent has had no opportunities
   to interfere with our $P_2$, while if a $P_3$ is completed, only
   one of the at two possible endpoints in the $P_3$ of the
   $C_5$-completing edge could be constrained by opponent activity,
   and if we have completed the $P_4$, there might be as many as two
   vertices constrained by the opponent in the $P_4$, but in all cases
   there are three vertices serving to construct a $C_5$-completing
   edge.

   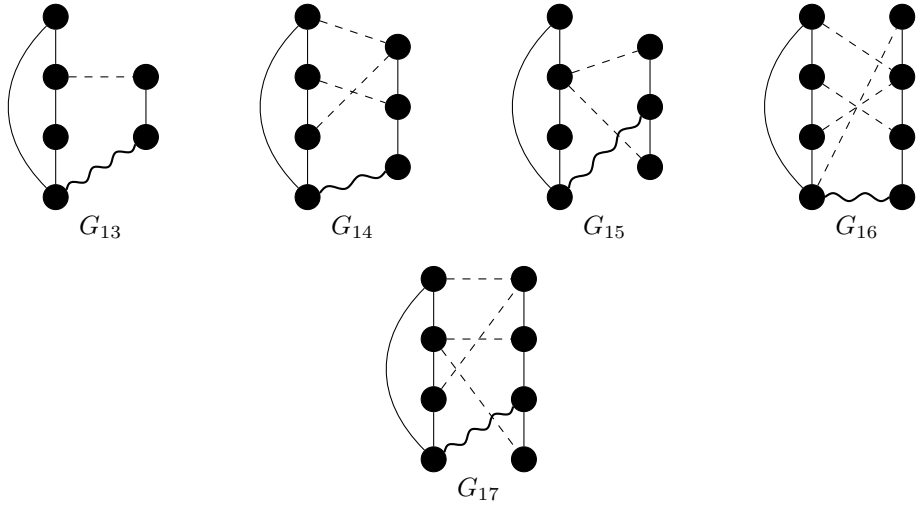
\begin{figure}
     \centering
     \begin{tikzpicture}[scale=0.8,nodes={shape=circle,fill,draw}]
       \node (u1) at (0,0) {};
       \node (u2) at (0,1) {};
       \node (u3) at (0,2) {};
       \node (u4) at (0,3) {};
       \draw (u1)--(u2)--(u3)--(u4)..controls (-1,2) and (-1,1)..(u1);
       
       \node (v1) at (1.5,1) {};
       \node (v2) at (1.5,2) {};

       \draw (v1)--(v2);
       \actionedge (u1)--(v1);
       \responseedge (u3)--(v2);
      \node[fill=none,draw=none] () at (0.75,-0.5) {$G_{13}$};
     \end{tikzpicture}
     \quad\quad\quad
     \begin{tikzpicture}[scale=0.8,nodes={shape=circle,fill,draw}]
       \node (u1) at (0,0) {};
       \node (u2) at (0,1) {};
       \node (u3) at (0,2) {};
       \node (u4) at (0,3) {};
       \draw (u1)--(u2)--(u3)--(u4)..controls (-1,2) and (-1,1)..(u1);
       
       \node (v1) at (1.5,0.5) {};
       \node (v2) at (1.5,1.5) {};
       \node (v3) at (1.5,2.5) {};

       \draw (v1)--(v2)--(v3);
       \actionedge (u1)--(v1);
       \responseedge (u3)--(v2);
       \responseedge (u2)--(v3);
       \responseedge (u4)--(v3);
      \node[fill=none,draw=none] () at (0.75,-0.5) {$G_{14}$};
     \end{tikzpicture}
     \quad\quad\quad
     \begin{tikzpicture}[scale=0.8,nodes={shape=circle,fill,draw}]
       \node (u1) at (0,0) {};
       \node (u2) at (0,1) {};
       \node (u3) at (0,2) {};
       \node (u4) at (0,3) {};
       \draw (u1)--(u2)--(u3)--(u4)..controls (-1,2) and (-1,1)..(u1);
       
       \node (v1) at (1.5,0.5) {};
       \node (v2) at (1.5,1.5) {};
       \node (v3) at (1.5,2.5) {};

       \draw (v1)--(v2)--(v3);
       \actionedge (u1)--(v2);
       \responseedge (u3)--(v1);
       \responseedge (u3)--(v3);
      \node[fill=none,draw=none] () at (0.75,-0.5) {$G_{15}$};
     \end{tikzpicture}
     \quad\quad\quad
     \begin{tikzpicture}[scale=0.8,nodes={shape=circle,fill,draw}]
       \node (u1) at (0,0) {};
       \node (u2) at (0,1) {};
       \node (u3) at (0,2) {};
       \node (u4) at (0,3) {};
       \draw (u1)--(u2)--(u3)--(u4)..controls (-1,2) and (-1,1)..(u1);
       
       \node (v1) at (1.5,0) {};
       \node (v2) at (1.5,1) {};
       \node (v3) at (1.5,2) {};
       \node (v4) at (1.5,3) {};

       \draw (v1)--(v2)--(v3)--(v4);
       \actionedge (u1)--(v1);
       \responseedge (u1)--(v4);
       \responseedge (u2)--(v3);
       \responseedge (u3)--(v2);
       \responseedge (u4)--(v3);
      \node[fill=none,draw=none] () at (0.75,-0.5) {$G_{16}$};
     \end{tikzpicture}
     \quad\quad\quad
     \begin{tikzpicture}[scale=0.8,nodes={shape=circle,fill,draw}]
       \node (u1) at (0,0) {};
       \node (u2) at (0,1) {};
       \node (u3) at (0,2) {};
       \node (u4) at (0,3) {};
       \draw (u1)--(u2)--(u3)--(u4)..controls (-1,2) and (-1,1)..(u1);
       
       \node (v1) at (1.5,0) {};
       \node (v2) at (1.5,1) {};
       \node (v3) at (1.5,2) {};
       \node (v4) at (1.5,3) {};

       \draw (v1)--(v2)--(v3)--(v4);
       \actionedge (u1)--(v2);
       \responseedge (u2)--(v4);
       \responseedge (u3)--(v3);
       \responseedge (u3)--(v1);
       \responseedge (u4)--(v4);
      \node[fill=none,draw=none] () at (0.75,-0.5) {$G_{17}$};
     \end{tikzpicture}
     \caption{Interruptions of $P_4$ construction after a $C_4$
       construction. The wavy edges indicate the opponent's interruptions,
       while dashed edges are ways to complete a $C_5$.}
    \label{fig:C4-and-P4-interrupted}
   \end{figure}
     
   \begin{figure}
     \centering
     \begin{tikzpicture}[scale=0.8,nodes={shape=circle,fill,draw}]
       \node (u1) at (0,0) {};
       \node (u2) at (0,1) {};
       \node (u3) at (0,2) {};
       \node (u4) at (0,3) {};
       \draw (u1)--(u2)--(u3)--(u4)..controls (-1,2) and (-1,1)..(u1);
       
       \vertexmarkone at (u1) {};
       \vertexmarkone at (u3) {};
       
       \vertexmarktwo at (u1) {};
       \vertexmarktwo at (u3) {};
       
       \vertexmarkthree at (u2) {};
       \vertexmarkthree at (u4) {};

       \node (v1) at (1.5,0) {};
       \node (v2) at (1.5,1) {};
       \node (v3) at (1.5,2) {};
       \node (v4) at (1.5,3) {};
       \draw (v1)--(v2)--(v3)--(v4);

       \vertexmarkthree at (v4) {};
       \vertexmarkone at (v3) {};
       \vertexmarktwo at (v2) {};
       \actionedge (u1)--(v1);
       \responseedge (u1)--(v4);
       \responseedge (u2)--(v3);
       \responseedge (u4)--(v3);
       \responseedge (u4)--(v1);
      \node[fill=none,draw=none] () at (0.75,-0.5) {$G_{18}$};
     \end{tikzpicture}
     \quad\quad\quad
     \begin{tikzpicture}[scale=0.8,nodes={shape=circle,fill,draw}]
       \node (u1) at (0,0) {};
       \node (u2) at (0,1) {};
       \node (u3) at (0,2) {};
       \node (u4) at (0,3) {};
       \draw (u1)--(u2)--(u3)--(u4)..controls (-1,2) and (-1,1)..(u1);
       
       \vertexmarkone at (u1) {};
       \vertexmarkone at (u3) {};
       
       \vertexmarktwo at (u2) {};
       \vertexmarktwo at (u4) {};
       
       \vertexmarkthree at (u1) {};
       \vertexmarkthree at (u3) {};

       \node (v1) at (1.5,0) {};
       \node (v2) at (1.5,1) {};
       \node (v3) at (1.5,2) {};
       \node (v4) at (1.5,3) {};
       \draw (v1)--(v2)--(v3)--(v4);

       \vertexmarkthree at (v4) {};
       \vertexmarkone at (v3) {};
       \vertexmarktwo at (v2) {};
       \actionedge (u1)--(v1);
       \responseedge (u1)--(v4);
       \responseedge (u3)--(v2);
       \responseedge (u4)--(v1);
      \node[fill=none,draw=none] () at (0.75,-0.5) {$G_{19}$};
     \end{tikzpicture}
     \quad\quad\quad
     \begin{tikzpicture}[scale=0.8,nodes={shape=circle,fill,draw}]
       \node (u1) at (0,0) {};
       \node (u2) at (0,1) {};
       \node (u3) at (0,2) {};
       \node (u4) at (0,3) {};
       \draw (u1)--(u2)--(u3)--(u4)..controls (-1,2) and (-1,1)..(u1);
       
       \vertexmarkone at (u2) {};
       \vertexmarkone at (u4) {};
       
       \vertexmarktwo at (u1) {};
       \vertexmarktwo at (u3) {};
       
       \vertexmarkthree at (u1) {};
       \vertexmarkthree at (u3) {};

       \node (v1) at (1.5,0) {};
       \node (v2) at (1.5,1) {};
       \node (v3) at (1.5,2) {};
       \node (v4) at (1.5,3) {};
       \draw (v1)--(v2)--(v3)--(v4);

       \vertexmarkthree at (v4) {};
       \vertexmarkone at (v3) {};
       \vertexmarktwo at (v2) {};
       \actionedge (u2)--(v1);
       \responseedge (u2)--(v4);
       \responseedge (u1)--(v3);
       \responseedge (u3)--(v3);
       \responseedge (u4)--(v2);
       \node[fill=none,draw=none] () at (0.75,-0.5) {$G_{20}$};
     \end{tikzpicture}
     \quad\quad\quad
     \begin{tikzpicture}[scale=0.8,nodes={shape=circle,fill,draw}]
       \node (u1) at (0,0) {};
       \node (u2) at (0,1) {};
       \node (u3) at (0,2) {};
       \node (u4) at (0,3) {};
       \draw (u1)--(u2)--(u3)--(u4)..controls (-1,2) and (-1,1)..(u1);
       
       \vertexmarkone at (u1) {};
       \vertexmarkone at (u3) {};
       
       \vertexmarktwo at (u2) {};
       \vertexmarktwo at (u4) {};
       
       \node (v1) at (1.5,0) {};
       \node (v2) at (1.5,1) {};
       \node (v3) at (1.5,2) {};
       \node (v4) at (1.5,3) {};
       \draw (v1)--(v2)--(v3)--(v4);

       \vertexmarktwo at (v4) {};
       \vertexmarkone at (v3) {};
       \vertexmarktwo at (v2) {};
       \actionedge (u1)--(v1);
       \responseedge (u1)--(v4);
       \responseedge (u2)--(v3);
       \responseedge (u3)--(v2);
       \responseedge (u4)--(v1);
      \node[fill=none,draw=none] () at (0.75,-0.5) {$G_{21}$};
     \end{tikzpicture}
     \quad\quad\quad
     \begin{tikzpicture}[scale=0.8,nodes={shape=circle,fill,draw}]
       \node (u1) at (0,0) {};
       \node (u2) at (0,1) {};
       \node (u3) at (0,2) {};
       \node (u4) at (0,3) {};
       \draw (u1)--(u2)--(u3)--(u4)..controls (-1,2) and (-1,1)..(u1);
       
       \vertexmarkthree at (u2) {};
       \vertexmarkthree at (u4) {};

       \vertexmarkone at (u2) {};
       \vertexmarkone at (u4) {};
       
       \vertexmarktwo at (u1) {};
       \vertexmarktwo at (u3) {};
       
       \node (v1) at (1.5,0) {};
       \node (v2) at (1.5,1) {};
       \node (v3) at (1.5,2) {};
       \node (v4) at (1.5,3) {};
       \draw (v1)--(v2)--(v3)--(v4);

       \vertexmarkthree at (v4) {};
       \vertexmarkone at (v3) {};
       \vertexmarktwo at (v3) {};
       \actionedge (u1)--(v1);
       \responseedge (u1)--(v4);
       \responseedge (u3)--(v2);
       \responseedge (u4)--(v1);
      \node[fill=none,draw=none] () at (0.75,-0.5) {$G_{22}$};
     \end{tikzpicture}
     \quad\quad\quad
     \begin{tikzpicture}[scale=0.8,nodes={shape=circle,fill,draw}]
       \node (u1) at (0,0) {};
       \node (u2) at (0,1) {};
       \node (u3) at (0,2) {};
       \node (u4) at (0,3) {};
       \draw (u1)--(u2)--(u3)--(u4)..controls (-1,2) and (-1,1)..(u1);
       
       \vertexmarkthree at (u2) {};
       \vertexmarkthree at (u4) {};

       \vertexmarkone at (u1) {};
       \vertexmarkone at (u3) {};
       
       \vertexmarktwo at (u1) {};
       \vertexmarktwo at (u3) {};
       
       \node (v1) at (1.5,0) {};
       \node (v2) at (1.5,1) {};
       \node (v3) at (1.5,2) {};
       \node (v4) at (1.5,3) {};
       \draw (v1)--(v2)--(v3)--(v4);

       \vertexmarkthree at (v4) {};
       \vertexmarkone at (v4) {};
       \vertexmarktwo at (v3) {};
       \actionedge (u1)--(v1);
       \responseedge (u2)--(v3);
       \responseedge (u3)--(v2);
       \responseedge (u4)--(v1);
      \node[fill=none,draw=none] () at (0.75,-0.5) {$G_{23}$};
     \end{tikzpicture}
     \quad\quad\quad
     \begin{tikzpicture}[scale=0.8,nodes={shape=circle,fill,draw}]
       \node (u1) at (0,0) {};
       \node (u2) at (0,1) {};
       \node (u3) at (0,2) {};
       \node (u4) at (0,3) {};
       \draw (u1)--(u2)--(u3)--(u4)..controls (-1,2) and (-1,1)..(u1);
       
       \vertexmarkone at (u2) {};
       \vertexmarkone at (u4) {};
       
       \vertexmarktwo at (u1) {};
       \vertexmarktwo at (u3) {};
       
       \node (v1) at (1.5,0) {};
       \node (v2) at (1.5,1) {};
       \node (v3) at (1.5,2) {};
       \node (v4) at (1.5,3) {};
       \draw (v1)--(v2)--(v3)--(v4);

       \vertexmarkone at (v4) {};
       \vertexmarkone at (v2) {};
       \vertexmarktwo at (v2) {};
       \actionedge (u1)--(v1);
       \responseedge (u1)--(v4);
       \responseedge (u2)--(v3);
       \responseedge (u4)--(v1);
      \node[fill=none,draw=none] () at (0.75,-0.5) {$G_{24}$};
     \end{tikzpicture}
     \quad\quad\quad
     \begin{tikzpicture}[scale=0.8,nodes={shape=circle,fill,draw}]
       \node (u1) at (0,0) {};
       \node (u2) at (0,1) {};
       \node (u3) at (0,2) {};
       \node (u4) at (0,3) {};
       \draw (u1)--(u2)--(u3)--(u4)..controls (-1,2) and (-1,1)..(u1);

       \vertexmarkone at (u2) {};
       \vertexmarkone at (u4) {};
       
       \vertexmarktwo at (u1) {};
       \vertexmarktwo at (u3) {};
       
       \node (v1) at (1.5,0) {};
       \node (v2) at (1.5,1) {};
       \node (v3) at (1.5,2) {};
       \node (v4) at (1.5,3) {};
       \draw (v1)--(v2)--(v3)--(v4);

       \vertexmarkone at (v4) {};
       \vertexmarkone at (v2) {};
       \vertexmarktwo at (v4) {};
       \actionedge (u1)--(v1);
       \responseedge (u2)--(v3);
       \responseedge (u3)--(v2);
       \responseedge (u4)--(v1);
      \node[fill=none,draw=none] () at (0.75,-0.5) {$G_{25}$};
     \end{tikzpicture}
     \caption{Responses to a $C_4$ construction, with the action we
       take to connect the two sections in wavy lines, and
       $C_5$-completing edges shown as dashed lines. Mutual
       adjacencies with outside vertices are denoted as in
       Figure~\ref{fig:P4-and-P3-with-constraints}.}
    \label{fig:C4-and-P4-with-constraints}
   \end{figure}
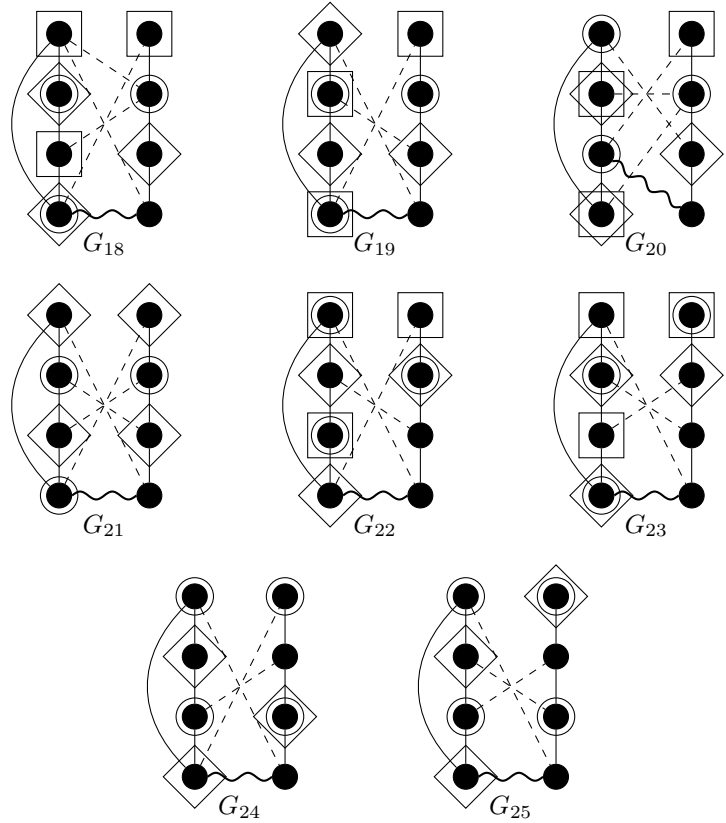

   If the opponent does not add an edge between the $C_4$ and the
   $P_4$, but engages in other interference, then we may craft our
   $P_4$ in such a manner as to be guaranteed that one of the outer
   vertices of the new $P_4$ is free of interference. Our optimal
   responses to each scenario are depicted in
   Figure~\ref{fig:C4-and-P4-with-constraints}. As was described in
   discussion of Figure~\ref{fig:P4-and-P3-with-constraints}, the
   $C_4$ may be arbitrarily highly adjacent to external vertices but
   the $P_4$ was constructed with initially clean vertices, and so at
   most three external adjacencies can be introduced by the opponent
   while we are building the $P_4$. In each case there are sufficiently
   many edges, with sufficiently disjoint sets of endpoints, that the
   opponent cannot, with the one move available after our edge is
   added, successfully make all of these $C_5$-completing moves
   invalid.

    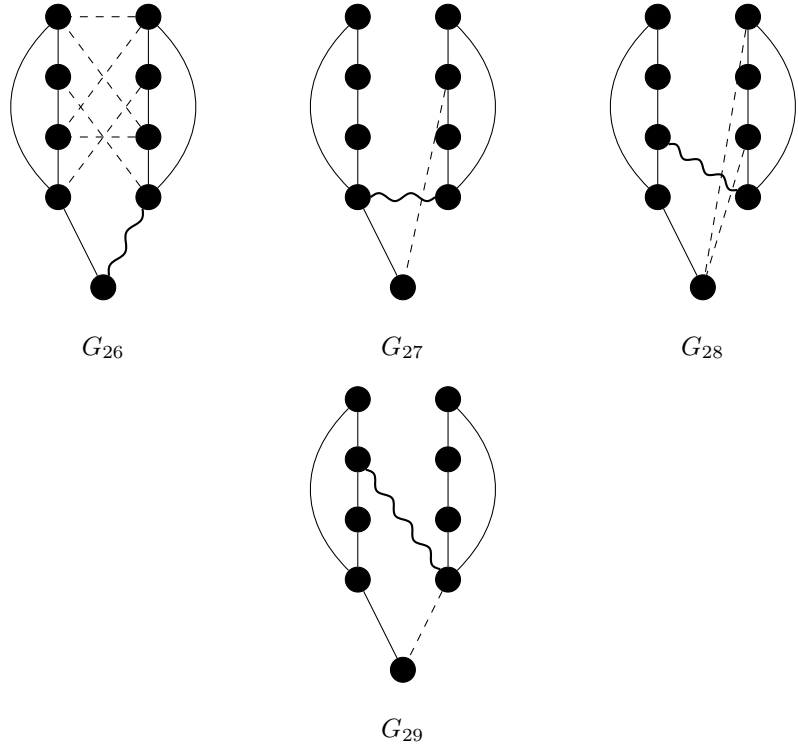
\begin{figure}
      \centering
      \begin{tikzpicture}[scale=0.8,nodes={shape=circle,fill,draw}]
        \node (u1) at (0,0) {};
        \node (u2) at (0,1) {};
        \node (u3) at (0,2) {};
        \node (u4) at (0,3) {};
        \draw (u1)--(u2)--(u3)--(u4)..controls (-1,2) and (-1,1)..(u1);
       
        \node (v1) at (1.5,0) {};
        \node (v2) at (1.5,1) {};
        \node (v3) at (1.5,2) {};
        \node (v4) at (1.5,3) {};
        \draw (v1)--(v2)--(v3)--(v4)..controls (2.5,2) and (2.5,1)..(v1);

        \node (w) at (0.75,-1.5) {};
        \draw (w)--(u1);

        \actionedge (w)--(v1);
        \responseedge (v1)--(u3);
        \responseedge (v2)--(u2);
        \responseedge (v2)--(u4);
        \responseedge (v3)--(u1);
        \responseedge (v4)--(u2);
        \responseedge (v4)--(u4);
      \node[fill=none,draw=none] () at (0.75,-2.5) {$G_{26}$};
      \end{tikzpicture}
      \quad\quad\quad
      \begin{tikzpicture}[scale=0.8,nodes={shape=circle,fill,draw}]
        \node (u1) at (0,0) {};
        \node (u2) at (0,1) {};
        \node (u3) at (0,2) {};
        \node (u4) at (0,3) {};
        \draw (u1)--(u2)--(u3)--(u4)..controls (-1,2) and (-1,1)..(u1);
       
        \node (v1) at (1.5,0) {};
        \node (v2) at (1.5,1) {};
        \node (v3) at (1.5,2) {};
        \node (v4) at (1.5,3) {};
        \draw (v1)--(v2)--(v3)--(v4)..controls (2.5,2) and (2.5,1)..(v1);

        \node (w) at (0.75,-1.5) {};
        \draw (w)--(u1);

        \actionedge (v1)--(u1);
        \responseedge (v3)--(w);
      \node[fill=none,draw=none] () at (0.75,-2.5) {$G_{27}$};
      \end{tikzpicture}
      \quad\quad\quad
      \begin{tikzpicture}[scale=0.8,nodes={shape=circle,fill,draw}]
        \node (u1) at (0,0) {};
        \node (u2) at (0,1) {};
        \node (u3) at (0,2) {};
        \node (u4) at (0,3) {};
        \draw (u1)--(u2)--(u3)--(u4)..controls (-1,2) and (-1,1)..(u1);
       
        \node (v1) at (1.5,0) {};
        \node (v2) at (1.5,1) {};
        \node (v3) at (1.5,2) {};
        \node (v4) at (1.5,3) {};
        \draw (v1)--(v2)--(v3)--(v4)..controls (2.5,2) and (2.5,1)..(v1);

        \node (w) at (0.75,-1.5) {};
        \draw (w)--(u1);

        \actionedge (v1)--(u2);
        \responseedge (v2)--(w);
        \responseedge (v4)--(w);
      \node[fill=none,draw=none] () at (0.75,-2.5) {$G_{28}$};
      \end{tikzpicture}
      \begin{tikzpicture}[scale=0.8,nodes={shape=circle,fill,draw}]
        \node (u1) at (0,0) {};
        \node (u2) at (0,1) {};
        \node (u3) at (0,2) {};
        \node (u4) at (0,3) {};
        \draw (u1)--(u2)--(u3)--(u4)..controls (-1,2) and (-1,1)..(u1);
       
        \node (v1) at (1.5,0) {};
        \node (v2) at (1.5,1) {};
        \node (v3) at (1.5,2) {};
        \node (v4) at (1.5,3) {};
        \draw (v1)--(v2)--(v3)--(v4)..controls (2.5,2) and (2.5,1)..(v1);

        \node (w) at (0.75,-1.5) {};
        \draw (w)--(u1);

        \actionedge (v1)--(u3);
        \responseedge (v1)--(w);
      \node[fill=none,draw=none] () at (0.75,-2.5) {$G_{29}$};
      \end{tikzpicture}
      \caption{When we add an additional vertex to a construction with
        two cycles, the opponent may place an edge, seen here with
        wavy lines. In each case, we may complete a $C_5$ with any of
        the dashed edges.}
      \label{fig:C4-and-C4-with-interruption}
    \end{figure}

    It now remains only to show that, if the opponent spends the last
    move after we construct this second $P_4$ converting it to a
    $C_4$, that we can nonetheless build a $C_5$. Doing so requires
    that we make use of one more new vertex, adding an edge between
    the new vertex and the most constrained vertex in the second $C_4$
    as seen in Figure~\ref{fig:C4-and-C4-with-interruption}. If the
    opponent responds by adding an edge among these vertices, we can
    always complete a $C_5$ immediately, since each of these cases has
    at least one $C_5$-completing edge which is not constrained: in
    the specfic case of $G_{26}$, there are edges between a large
    enough set of vertices that the two vertices the opponent has
    interfered with are insufficient to prevent one such edge from
    still being valid, while in the other cases, the $C_5$ completing
    edge has the known clean new point as an one of its endpoints. If
    instead the opponent does not provide an edge internal to the
    graph, then we know the opponent will have interfered at most
    three times, twice during the construction of the second $C_4$ and
    once after adding the most recent edge. These possibilities are
    seen in Figure~\ref{fig:C4-and-C4-with-constraints}, and in each
    case an optimal choice of move induces sufficiently independent
    $C_4$-completing edges that no single move by the opponent can
    obstruct them all. As was described in discussion of
    Figure~\ref{fig:P4-and-P3-with-constraints}, the first $C_4$ may
    be arbitrarily highly adjacent to external vertices but the second
    $C_4$ was constructed with initially clean vertices, and so at
    most two external adjacencies can be introduced by the opponent
    while we are building the $P_4$, with one move spent closing our
    $P_4$ into a $C_4$, and the opponent may then introduce one more
    adjacency while we introduce the ninth vertex.

    \begin{figure}
      \centering
      \begin{tikzpicture}[scale=0.65,nodes={shape=circle,fill,draw}]
        \node (u1) at (0,0) {};
        \node (u2) at (0,1) {};
        \node (u3) at (0,2) {};
        \node (u4) at (0,3) {};
        \draw (u1)--(u2)--(u3)--(u4)..controls (-1,2) and (-1,1)..(u1);
       
        \node (v1) at (1.5,0) {};
        \node (v2) at (1.5,1) {};
        \node (v3) at (1.5,2) {};
        \node (v4) at (1.5,3) {};
        \draw (v1)--(v2)--(v3)--(v4)..controls (2.5,2) and (2.5,1)..(v1);

        \node (w) at (0.75,-1.5) {};
        \draw (w)--(u1);

        \vertexmarkone at (u1) {};
        \vertexmarkone at (u3) {};
        \vertexmarkone at (v3) {};
       
        \vertexmarktwo at (u1) {};
        \vertexmarktwo at (u3) {};
        \vertexmarktwo at (v2) {};
        \vertexmarktwo at (v4) {};

        \actionedge (u3)--(v1);
        \responseedge (u2)--(v3);
        \responseedge (u4)--(v3);
        \responseedge (w)--(v1);
      \node[fill=none,draw=none] () at (0.75,-2.5) {$G_{30}$};
      \end{tikzpicture}
      \quad\quad\quad
      \begin{tikzpicture}[scale=0.65,nodes={shape=circle,fill,draw}]
        \node (u1) at (0,0) {};
        \node (u2) at (0,1) {};
        \node (u3) at (0,2) {};
        \node (u4) at (0,3) {};
        \draw (u1)--(u2)--(u3)--(u4)..controls (-1,2) and (-1,1)..(u1);
       
        \node (v1) at (1.5,0) {};
        \node (v2) at (1.5,1) {};
        \node (v3) at (1.5,2) {};
        \node (v4) at (1.5,3) {};
        \draw (v1)--(v2)--(v3)--(v4)..controls (2.5,2) and (2.5,1)..(v1);

        \node (w) at (0.75,-1.5) {};
        \draw (w)--(u1);

        \vertexmarkone at (w) {};
        \vertexmarkone at (u2) {};
        \vertexmarkone at (u4) {};
        \vertexmarkone at (v3) {};
       
        \vertexmarktwo at (w) {};
        \vertexmarktwo at (u2) {};
        \vertexmarktwo at (u4) {};
        \vertexmarktwo at (v2) {};
        \vertexmarktwo at (v4) {};

        \actionedge (u1)--(v3);
        \responseedge (u2)--(v1);
        \responseedge (u3)--(v4);
        \responseedge (u3)--(v2);
        \responseedge (u4)--(v1);
        \responseedge (w)--(v1);
      \node[fill=none,draw=none] () at (0.75,-2.5) {$G_{31}$};
      \end{tikzpicture}
      \quad\quad\quad
      \begin{tikzpicture}[scale=0.65,nodes={shape=circle,fill,draw}]
        \node (u1) at (0,0) {};
        \node (u2) at (0,1) {};
        \node (u3) at (0,2) {};
        \node (u4) at (0,3) {};
        \draw (u1)--(u2)--(u3)--(u4)..controls (-1,2) and (-1,1)..(u1);
       
        \node (v1) at (1.5,0) {};
        \node (v2) at (1.5,1) {};
        \node (v3) at (1.5,2) {};
        \node (v4) at (1.5,3) {};
        \draw (v1)--(v2)--(v3)--(v4)..controls (2.5,2) and (2.5,1)..(v1);

        \node (w) at (0.75,-1.5) {};
        \draw (w)--(u1);

        \vertexmarkone at (w) {};
        \vertexmarkone at (u2) {};
        \vertexmarkone at (u4) {};
        \vertexmarkone at (v3) {};
       
        \vertexmarktwo at (u1) {};
        \vertexmarktwo at (u3) {};
        \vertexmarktwo at (v2) {};

        \vertexmarkthree at (w) {};
        \vertexmarkthree at (u2) {};
        \vertexmarkthree at (u4) {};
        \vertexmarkthree at (v4) {};

        \actionedge (w)--(v1);
        \responseedge (v1)--(u3);
        \responseedge (v2)--(u2);
        \responseedge (v2)--(u4);
        \responseedge (v3)--(u1);
      \node[fill=none,draw=none] () at (0.75,-2.5) {$G_{32}$};
      \end{tikzpicture}
      \quad\quad\quad
      \begin{tikzpicture}[scale=0.65,nodes={shape=circle,fill,draw}]
        \node (u1) at (0,0) {};
        \node (u2) at (0,1) {};
        \node (u3) at (0,2) {};
        \node (u4) at (0,3) {};
        \draw (u1)--(u2)--(u3)--(u4)..controls (-1,2) and (-1,1)..(u1);
       
        \node (v1) at (1.5,0) {};
        \node (v2) at (1.5,1) {};
        \node (v3) at (1.5,2) {};
        \node (v4) at (1.5,3) {};
        \draw (v1)--(v2)--(v3)--(v4)..controls (2.5,2) and (2.5,1)..(v1);

        \node (w) at (0.75,-1.5) {};
        \draw (w)--(u1);

        \vertexmarkone at (u1) {};
        \vertexmarkone at (u3) {};
        \vertexmarkone at (v3) {};
       
        \vertexmarktwo at (w) {};
        \vertexmarktwo at (u2) {};
        \vertexmarktwo at (u4) {};
        \vertexmarktwo at (v2) {};

        \vertexmarkthree at (u1) {};
        \vertexmarkthree at (u3) {};
        \vertexmarkthree at (v4) {};

        \actionedge (w)--(v4);
        \responseedge (u1)--(v2);
        \responseedge (u2)--(v1);
        \responseedge (u2)--(v3);
        \responseedge (u4)--(v1);
        \responseedge (u4)--(v3);
      \node[fill=none,draw=none] () at (0.75,-2.5) {$G_{33}$};
      \end{tikzpicture}
      \quad\quad\quad
      \begin{tikzpicture}[scale=0.65,nodes={shape=circle,fill,draw}]
        \node (u1) at (0,0) {};
        \node (u2) at (0,1) {};
        \node (u3) at (0,2) {};
        \node (u4) at (0,3) {};
        \draw (u1)--(u2)--(u3)--(u4)..controls (-1,2) and (-1,1)..(u1);
       
        \node (v1) at (1.5,0) {};
        \node (v2) at (1.5,1) {};
        \node (v3) at (1.5,2) {};
        \node (v4) at (1.5,3) {};
        \draw (v1)--(v2)--(v3)--(v4)..controls (2.5,2) and (2.5,1)..(v1);

        \node (w) at (0.75,-1.5) {};
        \draw (w)--(u1);

        \vertexmarkone at (u1) {};
        \vertexmarkone at (u3) {};
        \vertexmarkone at (v3) {};
       
        \vertexmarktwo at (w) {};
        \vertexmarktwo at (u2) {};
        \vertexmarktwo at (u4) {};
        \vertexmarktwo at (v2) {};
        \vertexmarktwo at (v4) {};

        \actionedge (u3)--(v1);
        \responseedge (u1)--(v2);
        \responseedge (u1)--(v4);
        \responseedge (u2)--(v3);
        \responseedge (u4)--(v3);
        \responseedge (w)--(v1);
      \node[fill=none,draw=none] () at (0.75,-2.5) {$G_{34}$};
      \end{tikzpicture}
      \quad\quad\quad
      \begin{tikzpicture}[scale=0.65,nodes={shape=circle,fill,draw}]
        \node (u1) at (0,0) {};
        \node (u2) at (0,1) {};
        \node (u3) at (0,2) {};
        \node (u4) at (0,3) {};
        \draw (u1)--(u2)--(u3)--(u4)..controls (-1,2) and (-1,1)..(u1);
       
        \node (v1) at (1.5,0) {};
        \node (v2) at (1.5,1) {};
        \node (v3) at (1.5,2) {};
        \node (v4) at (1.5,3) {};
        \draw (v1)--(v2)--(v3)--(v4)..controls (2.5,2) and (2.5,1)..(v1);

        \node (w) at (0.75,-1.5) {};
        \draw (w)--(u1);

        \vertexmarkone at (w) {};
        \vertexmarkone at (u2) {};
        \vertexmarkone at (u4) {};
        \vertexmarkone at (v3) {};
       
        \vertexmarktwo at (u1) {};
        \vertexmarktwo at (u3) {};
        \vertexmarktwo at (v2) {};
        \vertexmarktwo at (v4) {};

        \actionedge (w)--(v1);
        \responseedge (v1)--(u3);
        \responseedge (v2)--(u2);
        \responseedge (v2)--(u4);
        \responseedge (v3)--(u1);
        \responseedge (v4)--(u2);
        \responseedge (v4)--(u4);
      \node[fill=none,draw=none] () at (0.75,-2.5) {$G_{35}$};
      \end{tikzpicture}
      \quad\quad\quad
      \begin{tikzpicture}[scale=0.65,nodes={shape=circle,fill,draw}]
        \node (u1) at (0,0) {};
        \node (u2) at (0,1) {};
        \node (u3) at (0,2) {};
        \node (u4) at (0,3) {};
        \draw (u1)--(u2)--(u3)--(u4)..controls (-1,2) and (-1,1)..(u1);
       
        \node (v1) at (1.5,0) {};
        \node (v2) at (1.5,1) {};
        \node (v3) at (1.5,2) {};
        \node (v4) at (1.5,3) {};
        \draw (v1)--(v2)--(v3)--(v4)..controls (2.5,2) and (2.5,1)..(v1);

        \node (w) at (0.75,-1.5) {};
        \draw (w)--(u1);

        \vertexmarkone at (w) {};
        \vertexmarkone at (u2) {};
        \vertexmarkone at (u4) {};
        \vertexmarkone at (v3) {};
       
        \vertexmarktwo at (u1) {};
        \vertexmarktwo at (u3) {};
        \vertexmarktwo at (v3) {};

        \vertexmarkthree at (w) {};
        \vertexmarkthree at (u2) {};
        \vertexmarkthree at (u4) {};
        \vertexmarkthree at (v4) {};

        \actionedge (w)--(v2);
        \responseedge (u1)--(v4);
        \responseedge (u2)--(v1);
        \responseedge (u3)--(v2);
      \node[fill=none,draw=none] () at (0.75,-2.5) {$G_{36}$};
      \end{tikzpicture}
      \quad\quad\quad
      \begin{tikzpicture}[scale=0.65,nodes={shape=circle,fill,draw}]
        \node (u1) at (0,0) {};
        \node (u2) at (0,1) {};
        \node (u3) at (0,2) {};
        \node (u4) at (0,3) {};
        \draw (u1)--(u2)--(u3)--(u4)..controls (-1,2) and (-1,1)..(u1);
       
        \node (v1) at (1.5,0) {};
        \node (v2) at (1.5,1) {};
        \node (v3) at (1.5,2) {};
        \node (v4) at (1.5,3) {};
        \draw (v1)--(v2)--(v3)--(v4)..controls (2.5,2) and (2.5,1)..(v1);

        \node (w) at (0.75,-1.5) {};
        \draw (w)--(u1);

        \vertexmarkone at (u1) {};
        \vertexmarkone at (u3) {};
        \vertexmarkone at (v3) {};
       
        \vertexmarktwo at (w) {};
        \vertexmarktwo at (u2) {};
        \vertexmarktwo at (u4) {};
        \vertexmarktwo at (v3) {};

        \vertexmarkthree at (u1) {};
        \vertexmarkthree at (u3) {};
        \vertexmarkthree at (v4) {};

        \actionedge (w)--(v1);
        \responseedge (u2)--(v2);
        \responseedge (u2)--(v4);
        \responseedge (u4)--(v2);
        \responseedge (u4)--(v4);
        \responseedge (u3)--(v1);
      \node[fill=none,draw=none] () at (0.75,-2.5) {$G_{37}$};
      \end{tikzpicture}
      \quad\quad\quad
      \begin{tikzpicture}[scale=0.65,nodes={shape=circle,fill,draw}]
        \node (u1) at (0,0) {};
        \node (u2) at (0,1) {};
        \node (u3) at (0,2) {};
        \node (u4) at (0,3) {};
        \draw (u1)--(u2)--(u3)--(u4)..controls (-1,2) and (-1,1)..(u1);
       
        \node (v1) at (1.5,0) {};
        \node (v2) at (1.5,1) {};
        \node (v3) at (1.5,2) {};
        \node (v4) at (1.5,3) {};
        \draw (v1)--(v2)--(v3)--(v4)..controls (2.5,2) and (2.5,1)..(v1);

        \node (w) at (0.75,-1.5) {};
        \draw (w)--(u1);

        \vertexmarkone at (w) {};
        \vertexmarkone at (u2) {};
        \vertexmarkone at (u4) {};
        \vertexmarkone at (v2) {};
       
        \vertexmarktwo at (u1) {};
        \vertexmarktwo at (u3) {};
        \vertexmarktwo at (v2) {};

        \vertexmarkthree at (w) {};
        \vertexmarkthree at (u2) {};
        \vertexmarkthree at (u4) {};
        \vertexmarkthree at (v4) {};

        \actionedge (w)--(v1);
        \responseedge (u1)--(v3);
        \responseedge (u3)--(v1);
      \node[fill=none,draw=none] () at (0.75,-2.5) {$G_{38}$};
      \end{tikzpicture}
      \quad\quad\quad
      \begin{tikzpicture}[scale=0.65,nodes={shape=circle,fill,draw}]
        \node (u1) at (0,0) {};
        \node (u2) at (0,1) {};
        \node (u3) at (0,2) {};
        \node (u4) at (0,3) {};
        \draw (u1)--(u2)--(u3)--(u4)..controls (-1,2) and (-1,1)..(u1);
       
        \node (v1) at (1.5,0) {};
        \node (v2) at (1.5,1) {};
        \node (v3) at (1.5,2) {};
        \node (v4) at (1.5,3) {};
        \draw (v1)--(v2)--(v3)--(v4)..controls (2.5,2) and (2.5,1)..(v1);

        \node (w) at (0.75,-1.5) {};
        \draw (w)--(u1);

        \vertexmarkone at (u1) {};
        \vertexmarkone at (u3) {};
        \vertexmarkone at (v2) {};
       
        \vertexmarktwo at (w) {};
        \vertexmarktwo at (u2) {};
        \vertexmarktwo at (u4) {};
        \vertexmarktwo at (v2) {};

        \vertexmarkthree at (u1) {};
        \vertexmarkthree at (u3) {};
        \vertexmarkthree at (v4) {};

        \actionedge (w)--(v1);
        \responseedge (u1)--(v3);
        \responseedge (u3)--(v1);
      \node[fill=none,draw=none] () at (0.75,-2.5) {$G_{39}$};
      \end{tikzpicture}
      \caption{Responses to a $C_4$ construction, if the opponent has
        turned our new $P_4$ into a $C_4$ as well. Wavy edges denote
        the correct move to make, and dashed lines indicate edges
        which will then complete a $C_5$. Mutual adjacencies with
        outside vertices are denoted as in
        Figure~\ref{fig:P4-and-P3-with-constraints}.}
      \label{fig:C4-and-C4-with-constraints}
    \end{figure}
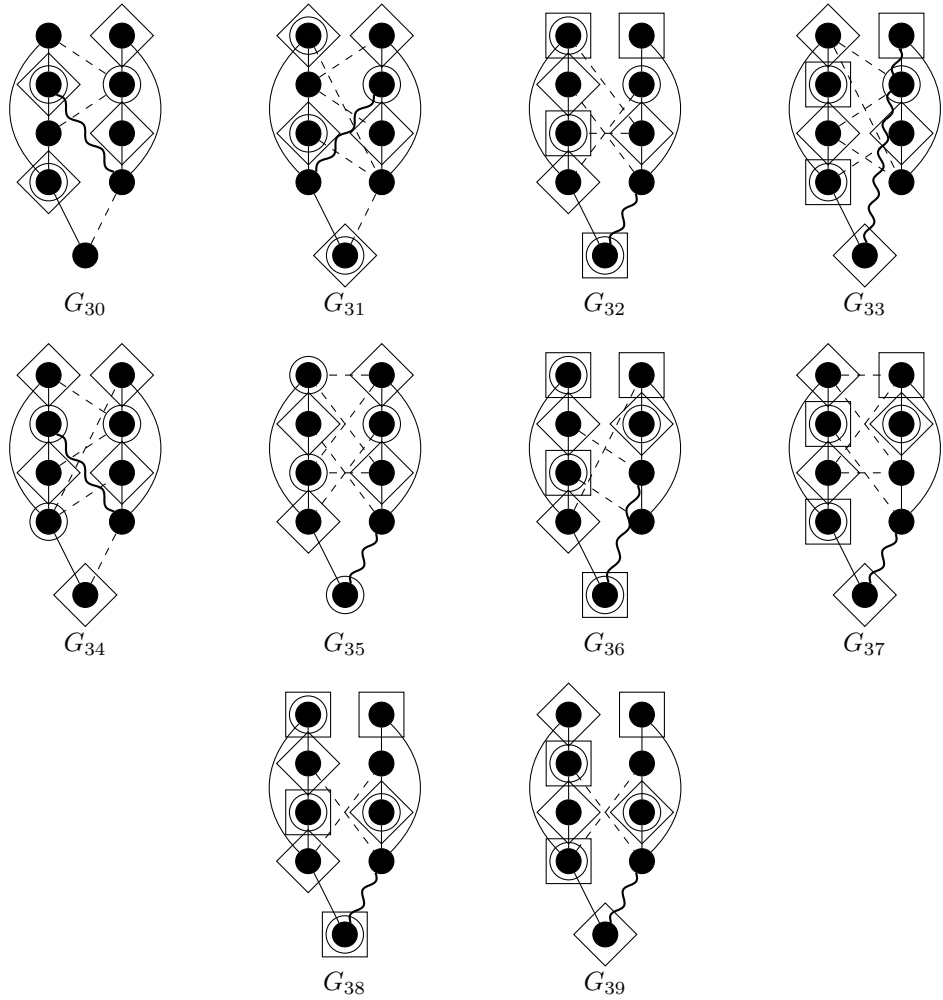

    \begin{table}
      \centering
      \begin{tabular}{|l|c|r|c|c||r|}\hline
        Graph&Comp. usage&Opp. moves&Structures&Dist. 3?&Count reduction\\\hline
        $G_1$, $G_2$&$5$&$4$&---&No&$9$\\
        $G_3$&$6$&$5$&$P_2$&No&$9$\\
        $G_4$, $G_5$, $G_6$&$6$&$5$&---&Yes&$10$\\
        $G_7$--$G_{12}$&$6$&$\leq 7$&$P_2$&No&$\leq11$\\
        $G_{13}$&$5$&$3$&---&No&$8$\\
        $G_{14}$&$6$&$4$&$P_2$&No&$8$\\
        $G_{15}$&$6$&$4$&---&Yes&$9$\\
        $G_{16}$&$7$&$5$&$P_3$&No&$8$\\
        $G_{16}$, w/ response&$7$&$4$&$P_2$&Yes&$\leq8$\\
        $G_{17}$&$7$&$5$&$P_2$&Yes&$9$\\
        $G_{18}$--$G_{25}$&$7$&$\leq 7$&$P_3$&No&$\leq10$\\
        $G_{23}$ or $G_{25}$, w/ response&$7$&$\leq 6$&$P_2$&Yes&$\leq10$\\
        $G_{26}$--$G_{29}$&$8$&$5$&$P_3$&Yes&$8$\\
        $G_{30}$--$G_{39}$&$8$&$\leq7$&$P_3$&Yes&$\leq 10$\\
        $G_{30}$--$G_{39}$, w/ response&$8$&$\leq6$&$P_2$&Yes&$\leq 11$\\\hline
      \end{tabular}
      \caption{Effect of given moves, and opponent's free moves, on the number of components remaining in the graph in each of the above 39 scenarios}
      \label{tab:costs}
    \end{table}

    The total reductions in count resulting from each of the 39 cases
    illustrated here are summarized in Table~\ref{tab:costs}. The
    ``component usage'' column counts the number of components which
    have been connected as a result of this procedure, which is simply
    one less than the number of vertices in the graph in question.
    Note that in practice the component usage may be reduced by
    distance-3 vertices or pre-existing structures in the last $C_5$
    construction; this effect has been factored into the count
    calculation.

    The ``opponent moves'' column counts the number of moves the
    opponent made which do not appear on the graph. This is ordinarily
    equal to the number of moves we have taken but may be reduced in
    cases where the graph only arises from an opponent move, as in the
    case where one of our paths has been closed into a $C_4$. Since
    the opponent moves might include joining components or distance-3
    vertices within a single component, each opponent move may result
    in a count reduction of 1. As was the case with the component
    usage calculation, the number of moves the opponent has may be
    reduced by the presence of pre-existing structures from the last
    $C_5$ construction; this too is factored into the count from that
    construction. Several of the graphs include the presumption that
    opponent moves may be spent interferentially, and in these graphs
    only an upper bound on the number of possible unconstrained
    opponent moves is included, since the opponent may spend several
    moves in interference.

    The ``structures'' and ``distance 3'' columns describe the
    elements which remain after the construction procedure and which
    will reduce the cost of the next construction. In particular, any
    $P_2$ and $P_3$ subgraphs which are left in $U$ are described
    here, as well as if a single component is guaranteed to contain
    two vertices at a distance of 3.

    Under certain circumstances, the opponent's actions may affect the
    remaining structures which we are able to leave behind. Thus, with
    a reduction in the opponent's unconstrained moves, they may force
    a less desirable set of remaining structures. In particular,
    graphs $G_{16}$, $G_{23}$, $G_{25}$, and $G_{30}$ through $G_{39}$
    are subject to such interference. These scenarios are handled
    separately in Table~\ref{tab:costs}. With all of these
    quantifications of the 37 different graphs, the effect of each
    scenario on the total count is easily calculated to be the sum of
    the component usage and the number of the opponent's unconstrained
    moves, reduced by whatever count bonus accrues from leftover
    structures and distance properties in $U$. It is thus easily
    observed that no scenario reduces the count by more than 11.
\end{proof}

Once we have constructed these $\lfloor\frac{n-2}{11}\rfloor$ cycles
of length 5, it is easy to show that, regardless of the following
moves, the final edge-density is significantly less than the
balanced-bipartite bound of $\frac14$.


\begin{theorem}
\begin{gather*}
\sat_g(\cf;n)\leq\frac{26}{121}n^2+o(n^2)\\
\sat'_g(\cf;n)\leq\frac{26}{121}n^2+o(n^2)\\
\end{gather*}
\end{theorem}
\begin{proof}
  Given $n$ vertices, Theorem~\ref{theorem:C5-construction}
  guarantees that we may construct approximately $\frac n{11}$ cycles
  of length 5, leaving approximately $\frac{6n}{11}$ vertices not
  incorporated into cycles. Between any two $C_5$s in a triangle-free graph there are 10 or fewer edges, so the number of edges among the
  approximately $\frac n{11}$ cycles is no more than
  $$5\cdot\frac n{11}+10\cdot\binom{\frac n{11}}2+o(n^2)=\frac{5}{121}n^2+o(n^2)$$
  while among the remaining $\frac{6n}{11}$ vertices, the densest the
  edges can possibly be is in a balanced complete bipartite graph,
  which has approximately
  $\frac14\left(\frac{6n}{11}\right)^2=\frac{9}{121}n^2$ edges, and between
  these $\frac{6n}{11}$ remaining vertices and the cycles, we know
  each cycle can have at most two vertices adjacent to a single
  vertex, yielding
  $2\cdot\frac{6n}{11}\cdot\frac{n}{11}=\frac{12}{121}n^2$ edges, for a
  total edge density of
  $$\frac{5}{121}+\frac{9}{121}+\frac{12}{121}=\frac{26}{121}$$
\end{proof}
\section{Potential improvements}
The above-determined edge density is still slightly above the
previously cited bound of $\frac15$, although it represents a
significant improvement on the trivial bound of $\frac14$. Certain
aspects of the above proof and underlying strategy, however, may be
amenable to improvement. In particular, the above result does not
include any strategic choices made after the completion of as many
$C_5$ subgraphs as possible, but merely assumes worst-case results
regarding the introduction of additional edges.

One aspect of the above method which may lend itself to improvement is
that the term $10\cdot\binom{\frac n{11}}2$ results from the
assumption that all except a subquadratic number of pairs of $C_5$
subgraphs have ten edges among them. After the $C_5$-construction
phase, since only a linear number of edges have been added, almost all
of these pairs will have no edges among them, and it may be possible
to ``spoil'' the prospect of adding ten edges among most of the
pairs. We may note that, up to isomorphism, there is only one way to
add ten edges among a pair of $C_5$s without forming a triangle. There
are 10 different ways to orient this structure on a pair of labeled
$C_5$s, and a specific orientation may be determined by as few as two
edges. Thus, by adding edges judiciously between two $C_5$s, it is
likely possible to force the opponent, if they are attempting to build
such a structure of ten edges, to include several specific edges,
which we could then make moves chosen to prohibit. It thus seems
likely that in a positive fraction of the pairs --- or quite possibly
in almost all of the pairs --- an edge density of 10 could be
prevented. Such an improvement could reduce this term from
$10\cdot\binom{\frac n{11}}2$ to $9\cdot\binom{\frac n{11}}2$,
effecting a reduction of $\frac1{242}$ in the edge density achieved. A
further reduction by a strategy limiting edge density to 8 or less (a
maximal triangle-free set of adjacencies between two $C_5$s could use
as few as 5 edges) could serve to reduce the density further, but
unfortunately doing so is likely to be very difficult. In reducing
edge density from 10 to 9, we would be aided by the fact that a unique
configuration is necessary to achieve 10 edges, whereas 9 edges can be
achieved in many different ways.

It may also be possible to improve the efficiency of the
$C_5$-building process. As was seen in Table~\ref{tab:costs}, a cost
of 11 when building a $C_5$ emerges only if the opponent is
interfering minimally with our construction and is instead building
edges elsewhere. It is possible that wise utilization of the vertices
on which the opponent is building these edges might serve to bring our
efficiency up to an ability to build approximately $\frac n{10}$
cycles of length 5. Such an improvement would reduce the edge density
to $\frac{17}{80}$, or, if combined with the above edge-density
improvement, $\frac{83}{400}$. Even this hoped-for improvement,
however, is still slightly higher than the sought-after $\frac15$
edge-density bound, which may not be achievable even by refinements of
this strategy.

\bibliographystyle{amsplain}
\bibliography{bib,extra}

\end{document}